\documentclass[12pt]{article}

\usepackage{amsmath}
\usepackage{amssymb}
\usepackage{amsthm}
\usepackage{graphicx}
\usepackage{amscd}
\usepackage{epic, eepic}
\usepackage{url}
\usepackage{subcaption}
\usepackage{color}
\usepackage[utf8]{inputenc} 

\usepackage{enumerate}

\newtheorem{theorem}{\bf Theorem}[section]
\newtheorem{lemma}[theorem]{\bf Lemma}
\newtheorem{cor}[theorem]{\bf Corollary}
\newtheorem{proposition}[theorem]{\bf Proposition}
\newtheorem{problem}[theorem]{\bf Problem}
\newtheorem{prop}[theorem]{\bf Proposition}
\newtheorem{conj}[theorem]{\bf Conjecture}

\newtheorem{remark}[theorem]{\bf Remark}

\usepackage[colorlinks=true, allcolors=blue]{hyperref}

\title{Triangle areas in line arrangements }
\author{Gábor Damásdi \thanks{ MTA-ELTE Lend\"ulet Combinatorial Geometry Research Group, Institute of Mathematics, E\"otv\"os Lor\'and University (ELTE), Budapest, Hungary. Supported by the \'UNKP-18-3 New National Excellence Program of the Ministry of Human Capacities  \tt{damasdigabor@caesar.elte.hu}} \and
Leonardo Mart\'inez-Sandoval  \thanks{Sorbonne Université, Institut de Mathématiques de Jussieu – Paris Rive Gauche (UMR 7586) Paris, France. Supported by the grant ANR-17-CE40-0018 of the French National 
Research Agency ANR (project CAPPS) \tt{leomtz@im.unam.mx}} \and
Dániel T. Nagy  \thanks{Alfr\'ed R\'enyi Institute of Mathematics, P.O.B. 127, Budapest H-1364, Hungary. Research supported by National Research, Development and Innovation Office - NKFIH grants K 116769, K 132696 and FK 132060. \tt{ nagydani@renyi.hu}} \and
 Zolt\'an L\'or\'ant Nagy \thanks{MTA-ELTE Geometric and Algebraic Combinatorics Research Group and  E\"otv\"os Lor\'and University, Department of Geometry, Budapest, Hungary, H--1117 Budapest, P\'azm\'any P.\ s\'et\'any 1/C. 
 Supported by the Hungarian Research Grant (OTKA) No. K 120154 and by the János Bolyai Scholarship of the Hungarian Academy of Sciences. \tt{ nagyzoli@cs.elte.hu}}}
\date{} 

\begin{document}
\maketitle

\begin{abstract}

A widely investigated subject in combinatorial geometry, originated from Erdős,  is the following. Given a point set $P$ of cardinality $n$ in the plane, how can we describe the distribution of the determined distances? This has been generalized in many directions.

In this paper we propose the following variants. What is the maximum number of triangles of unit area, maximum area or minimum area, that can be determined by an arrangement of $n$ planar lines? 

We prove that the order of magnitude for the maximum occurrence of unit areas lies between $\Omega(n^2)$ and $O(n^{9/4})$. This result is strongly connected to additive combinatorial results and Szemerédi--Trotter type incidence theorems. 
Next we show an almost tight bound for the maximum number of minimum area triangles. 
Finally we present lower and upper bounds for the  maximum area and distinct area problems by combining algebraic, geometric and combinatorial techniques.

\end{abstract}

\section{Introduction}

There are number of interesting question originating from Erdős, concerning distances in planar point sets \cite{EP46}. He asked to determine the maximum number of equal distances that $n$ planar points can form, the minimum number of distinct distances they can form and the maximum number of appearances of the largest/smallest distance.  He also considered how large subset is guaranteed to exists in a point set such that the distances within that subset are all distinct. 

Erdős and Purdy also studied the related problem of the maximum number of occurrences
of the same area among the triangles determined by $n$ points in the plane \cite{EP71}.  Since then, several variants has been established and the former results of Erdős and Purdy have been settled for some cases, see e.g. \cite{BRS01, EP71, RS17}. 


In this paper we consider the following variants of the original problem which can be considered as the dual setting. We are given $n$ lines on the Euclidean plane and we are seeking for conditions on the distribution of the areas of triangles formed by the triples of lines. More precisely, we investigate the following four main problems and compare the results to the corresponding problems concerning triples of points.

\begin{problem}\label{Pr_uni}  Determine the largest possible number $f(n)$ of triangles of unit area formed by $n$ lines in the Euclidean plane.
\end{problem}


\begin{problem}\label{Pr_min} Determine the largest possible number $m(n)$ of triangles having minimum area formed by $n$ lines in the Euclidean plane.
\end{problem}

\begin{problem}\label{Pr_max}  Determine the largest possible number $M(n)$ of triangles with maximum area formed by $n$ lines in the Euclidean plane.
\end{problem}

\begin{problem}\label{Pr_dis}  Determine the largest possible number $D(n)$ such that in any arrangement of n lines (satisfying some generality conditions) there are D(n) lines that form triangles of distinct areas.
\end{problem}

Concerning these problems, we achieved the following results.
\bigskip

\begin{theorem}\label{unitarea}
For the maximum number of triangles of unit area, we have $$f(n)= O\left(n^{\frac{9}{4}+\varepsilon}\right)$$ for every fixed $\varepsilon >0$, while  $f(n)=\Omega(n^2)$.
\end{theorem}

\begin{theorem}\label{minim}
 $$ \left\lfloor \frac{n^2-n}{6} \right\rfloor \leq m(n)\leq  \left\lfloor \frac{n^2-2n}{3}\right\rfloor$$ holds for the occurrences of the minimum area, if $n\geq 6$.
\end{theorem}

\begin{theorem}\label{maxi}
For the maximum number of triangles of maximum area, we have $$\frac{7}{5}n-O(1)< M(n) < \frac{2}{3}n(n-2).$$
\end{theorem}

\begin{theorem}\label{disti}
For the largest subset of lines forming triangles of distinct areas, we have $$n^{\frac{1}{5}} < D(n) ,$$
provided that there are no six lines that are tangent to a common conic.
\end{theorem}

If one wishes to find a large subset of lines defining distinct area triangles, it is necessary
to make some additional assumptions about the set of lines we are considering. The most natural
assumption is to suppose that there are no parallel lines in the set, and no three of them through a common point. However, we were not able to obtain non-trivial bounds under these assumptions. Since $5$ lines always have a common tangent conic, another natural general position assumption is to require that no $6$ of them do, as in the hypothesis of Theorem \ref{disti}.


To put these results into perspective, let us recall a related problem, first asked by Oppenheim in 1967, which reads as follows: What is the maximum number of
triangles of unit area that can be determined by $n$ points
in the plane? The first breakthrough after the investigation of Erdős and Purdy \cite{EP71} was due to Pach and Sharir \cite{PS92}, who obtained an upper bound $O(n^{2+1/3})$ via a Szemerédi-Trotter type argument. This bound was improved in \cite{AS10, dumtoth07} and recently by Raz and Sharir to $O(n^{2+2/9})$ in \cite{RS17}. Here the lower bound is a simple lattice construction from \cite{EP71}, yielding $\Omega(n^2\log\log{n})$. Our Theorem \ref{unitarea} also indicates that the straightforward application of some Szemerédi-Trotter type result can be improved. However, in the next Section we will point out that in some relaxation, it would provide the right order of magnitude.

As in the case of counting equal distances, the minimum and maximum area problems determined by point sets turned out to be easier, and they were asymptotically settled by Brass, Rote and Swanepoel \cite{BRS01}. The minimum area problem was later refined in \cite{dumtoth07}. Concerning the occurrences of the maximum area, the upper bound happens to be exactly $n$. This is a rather common phenomenon in this field, we could mention the well-known theorem of Hopf and Pannwitz and similar results, see \cite{book}. Surprisingly, Theorem \ref{maxi} shows that this is not the case in our problem.

The problem of the largest subset of points with distinct pairwise distances was originally posed by Erd\H{o}s \cite{EP57} and generalised recently to distinct $k$-dimensional volumes in $\mathbb{R}^d$ by Conlon et al. \cite{Con}. For a point of comparison, they note that in the case $d=2$, $\Omega(n^{1/5})$ points can be chosen from a set of $n$ points in general position so that each triple determines a triangle of distinct area. The best upper bound so far is attained by choosing $\Omega(n)$ points in general position on the $n\times n$ grid. From Pick's Theorem \cite{Pick} it is clear that twice the area of a lattice triangle is an integer. Therefore lattice triangles on this grid define at most $O(n^2)$ areas, so the upper bound for the problem is $O(n^{2/3})$.

The paper is built up as follows. In Section 2 we discuss Problem \ref{Pr_uni} and prove Theorem \ref{unitarea}. In order to do this, we consider first the maximum number of unit area triangles lying on a fixed line, and prove tight results up to a constant factor. Then we will apply a deep result of Pach and Zahl to complete the proof of our main theorem.\\
Section 3 is devoted to the Problems \ref{Pr_min} and  \ref{Pr_max}, and we prove Theorem \ref{minim} and  \ref{maxi}. Section 4 concerns Problem \ref{Pr_dis} and contains the proof of Theorem \ref{disti}. Finally we discuss some related problems and open questions in Section 5.

\section{The number of unit area triangles}

\subsection{Unit area triangles on a single line}
A natural way to give an upper bound on $f(n)$ is to consider how many of the unit area triangles can be supported by a fixed line. Then $f(n)$ is at most $n/3$ times larger. 
\begin{problem}
Let $\ell$ be a line and  $\mathcal{L}$ be a set of $n$ lines 
and consider the triangles formed by $\ell$ and two elements of $\mathcal{L}$. Determine the largest possible number $g(n)$ of triangles of unit area among these.  
\end{problem}

We determine the order of magnitude of $g(n)$ by turning the problem into an incidence problem for points and lines.

\begin{theorem}\label{online}
For the maximum number of triangles of unit area having a common supporting line $\ell$, $g(n)=\Theta(n^{4/3})$ holds.
\end{theorem}

We may assume that $\ell$ is horizontal, and that the rest of the lines in $\mathcal{L}=\{\ell_1,\ldots,\ell_n\}$ are not horizontal lines. Let $x_i$ denote the $x$-coordinate of the intersection of $\ell$ and $\ell_i$ and let $y_i=\cot{\alpha_i}$ where $\alpha_i$ denotes the (directed) angle determined by $\ell$ and $\ell_i$. Let $T_{ij}$ denote the triangle formed by $\ell,\ell_i$ and $\ell_j$. Notice that the parameters $(x,y)$ provide an exact description of any line not parallel to $\ell$, while a parallel line $\ell'\parallel \ell$ would not contribute to the number of unit area triangles supported by $\ell$. Let us denote by $e(x,y)$ the line described by parameters $(x,y)$.

\begin{lemma}\label{areaoftri}
Assume that $x_i\not=x_j$ and $y_i\not=y_j$. The area of triangle $T_{ij}$ is $$\text{Area}(T_{ij})=\frac{(x_j-x_i)^2}{2|(y_i-y_j)|}.$$ 
\end{lemma}

\begin{proof}
The equation of the lines $e(x_i, y_i)$ and $e(x_j, y_j)$ is $y=\frac{x-x_i}{y_i}$ and $y=\frac{x-x_j}{y_j}$ respectively. Therefore their intersection point is $(x,y)=\left(\frac{x_jy_i-x_iy_j}{y_i-y_j}, \frac{x_j-x_i}{y_i-y_j}\right)$.
$$\text{Area}(T_{ij})=\frac{1}{2}|x_j-x_i|\left|\frac{x_j-x_i}{y_i-y_j}\right|=\frac{(x_j-x_i)^2}{2|(y_i-y_j)|}$$
\end{proof}

\begin{proof}[Proof of Theorem \ref{online}]

 We apply the lemma above. Supposing that $y_i>y_j$,  $T_{ij}$ is of unit area if and only if $2y_i-x_i^2=-2x_ix_j+x_j^2+2y_j$. In other words $T_{ij}$ is of unit area if and only if the point $(x_i,2y_i-x_i^2)$ lies on the line $y=-2x_jx+2y_j+x_j^2$. 

By the Szemerédi--Trotter theorem, $n$ lines and $n$ points have $O(n^{4/3})$ incidences. Applying this to the lines $y=-2x_jx+2y_j+x_j^2$ and the points $(x_i,2y_i-x_i^2)$ we get $g(n)=O(n^{4/3})$. 

On the other hand there exists $n/2$ lines and $n/2$ points that have $\Omega (n^{4/3})$ incidences. We can write these points in the form $(x_i,2y_i-x_i^2)$ for some $(x_1,y_1),\dots, (x_{n/2},y_{n/2})$. Similarly we can write the lines in the form $y=-2x_jx+2y_j+x_j^2$ for some $(x_{n/2+1},y_{n/2+1}),\dots, (x_{n},y_{n})$. Then the $n$ lines given by the assignment $(x_i,y_i) \rightarrow e(x_i,y_i)$  determine $\Omega (n^{4/3})$ unit area triangles. Therefore $g(n)=\Theta(n^{4/3})$.   
\end{proof}
Let us mention that the upper bound is also implied by the powerful theorem of Pach and Sharir \cite{PS98}.

\begin{theorem}[\cite{PS98}]\label{incidence}
 Let $P$ be a set of $m$ points and let $\Gamma$ be a set of $n$ distinct irreducible algebraic curves of degree at most $k$, both in $\mathbb{R}^2$. If the incidence graph of $P \times \Gamma$ contains no copy of $K_{s,t}$, then the number of incidences is $$O(m^{\frac{s}{2s-1}}n^{\frac{2s-2}{2s-1}}+m+n).$$
\end{theorem}

Indeed, the lines  were described by their parameters $(x_1,y_1),\dots,$ $(x_n,y_n)$, and consider the unit parabolas $2y=x^2-2xx_j+x_j^2-2y_j$ $(j=1, \ldots, n)$. The $i$th point lies on the $j$th parabola if and only if the triangle $T_{ij}$ has unit area. A unit parabola is determined by two of its points, so the incidence graph doesn't contain $K_{2,2}$ and Theorem \ref{incidence} can be applied for $s=t=2$. Since we have $n$ points and $n$ curves, the number of unit area triangles having a common supporting line is $O(n^{2/3}n^{2/3})=O(n^{4/3})$.  

\begin{cor}\label{upperline}
The bound above yields $f(n)=O(n^{7/3})$ for the maximum number of unit area triangles.
\end{cor}

\subsection{Upper bound on the maximum number of unit area triangles}

\subsubsection{Reformulation in   additive combinatorics}

For any arrangement of $n$ lines, an $y$-axis can be chosen such that the parameters of these lines as introduced above in Lemma \ref{areaoftri}, satisfy $x_1\le x_2 \le \dots \le x_n$ and $y_1\ge y_2 \ge \dots \ge y_n$. Let $H=\{(x_i, y_i) \mid 1\leq i\leq n\}\subseteq \mathbb{R}^2$.

\begin{prop}
The number of unit-area triangles in the arrangement equals the number of solutions in $H{\times} H{\times} H$ to the (rational) equation
\begin{equation}\label{xytriple}\frac{(x_j-x_i)^2}{y_i-y_j}+\frac{(x_k-x_j)^2}{y_j-y_k}+\frac{(x_i-x_k)^2}{y_k-y_i}=2.\end{equation} 
\end{prop}




\begin{proof}

We may assume by rotation that none of the $n$ lines are horizontal, and consider a horizontal line $\ell$ located under all the intersections of the $n$ lines. Taking $\ell$ as the $x$ axis of a coordinate system, we let $x_i$ to be the coordinate of the $i$-th intersection of $\ell$ with another line (which we denote by $\ell_i$). Let $\alpha_i$ denote the (directed) angle appearing between $\ell$ and $\ell_i$, see Figure \ref{figure1}. Let $y_i=\cot{\alpha_i}$. Since there are no intersections under or on $\ell$, we have $\alpha_1\le \alpha_2\le \dots \le \alpha_n$ and therefore $y_1\ge y_2\ge\dots \ge y_n$.

\begin{figure}[h]
\begin{center}
\includegraphics[scale = 1.2] {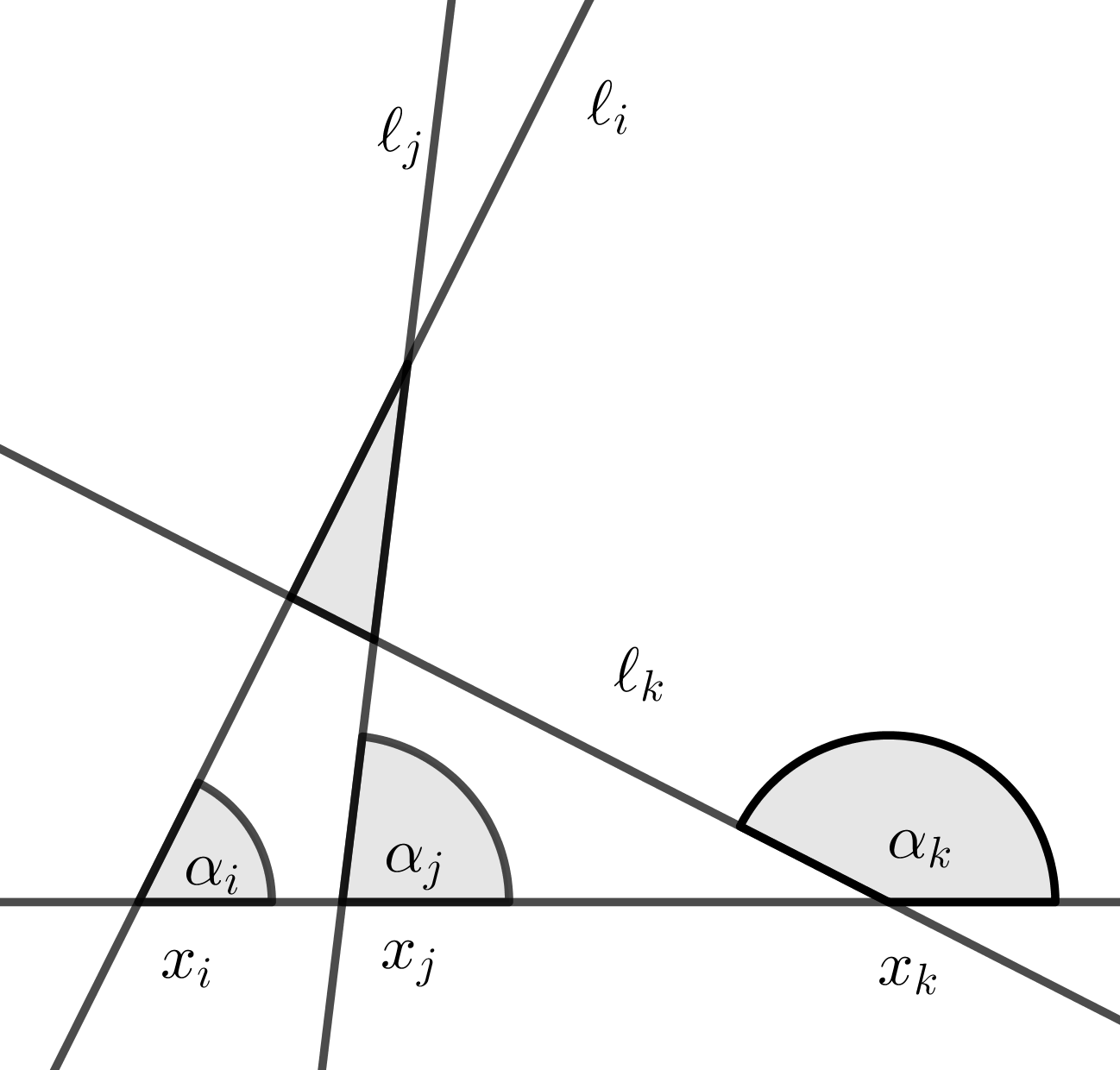}
\includegraphics[scale = 1.2] {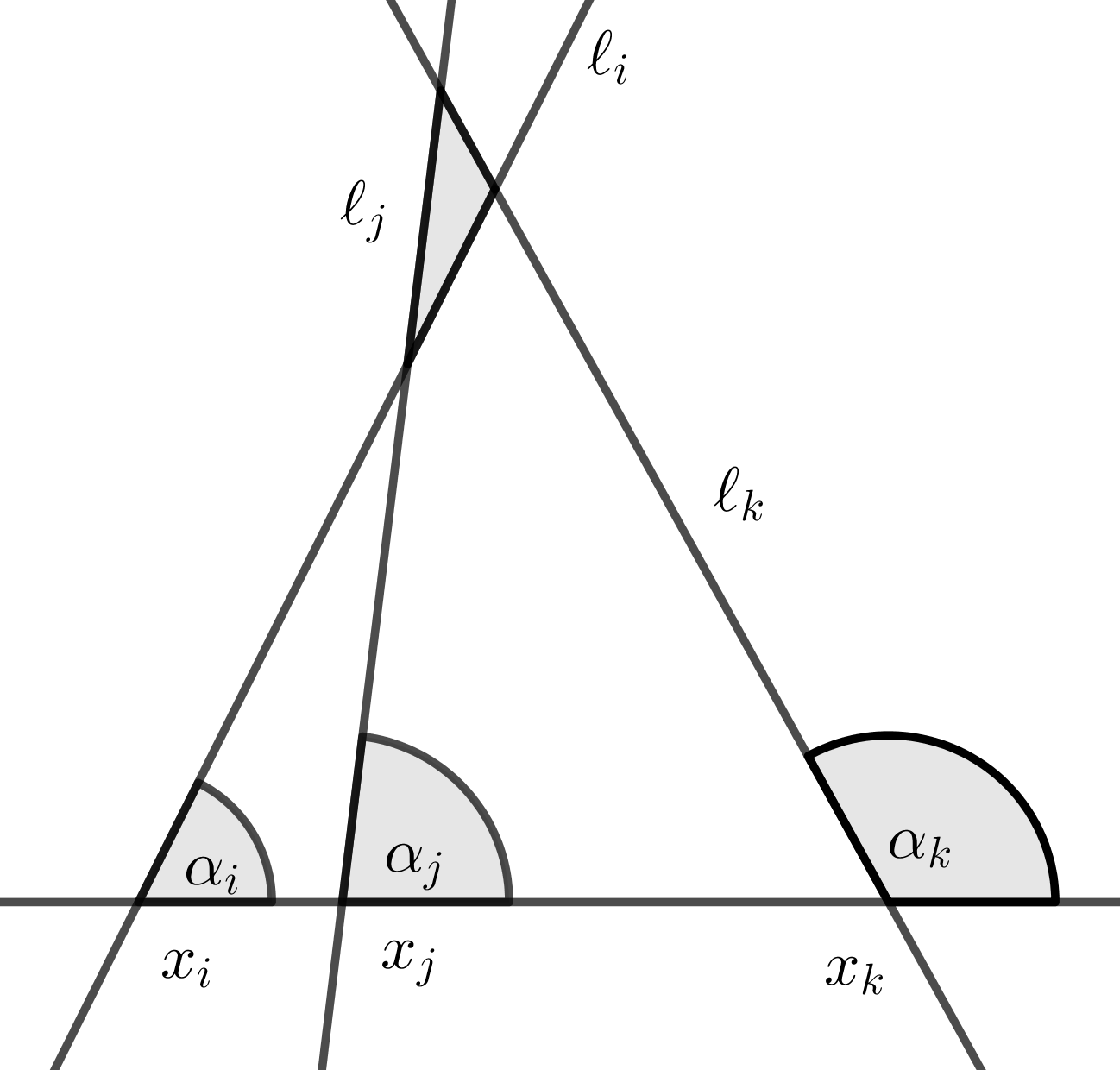}
\end{center}
\caption{The calculation of the triangle area, formed by the lines $\ell_i$, $\ell_j$ and $\ell_k$}
\label{figure1}\end{figure}

Since $y_i\ge y_j$, the area of the triangle $T_{ij}$ determined by $\ell$, $\ell_i$ and $\ell_j$ is $\frac{(x_j-x_i)^2}{|2(y_i-y_j)|}=\frac{(x_j-x_i)^2}{2(y_i-y_j)}$. (If $y_i=y_j$ then $\ell_i$ and $\ell_j$ are parallel, and they form no triangle.) The area of the triangle determined by the lines $\ell_i$, $\ell_j$ and $\ell_k$ can be calculated as
$$Area(T_{ij})+Area(T_{jk})-Area(T_{ik})=\frac{(x_j-x_i)^2}{2(y_i-y_j)}+\frac{(x_k-x_j)^2}{2(y_j-y_k)}-\frac{(x_k-x_i)^2}{2(y_i-y_k)}=$$
$$\frac{(x_j-x_i)^2}{2(y_i-y_j)}+\frac{(x_k-x_j)^2}{2(y_j-y_k)}+\frac{(x_i-x_k)^2}{2(y_k-y_i)}.$$

Therefore the problem of finding an arrangement of lines determining $f(n)$ triangles of unit area is equivalent to finding some reals $x_1<x_2<\dots <x_n$ and $y_1\ge y_2\ge\dots \ge y_n$ such that \eqref{xytriple} is satisfied for the maximal number of index triples.
\end{proof}

\subsubsection{Improved upper bound for $f(n)$}

We improve here the bound achieved by Corollary \ref{upperline}.  To do this, we recall a recent result of Sharir and Zahl \cite{SZ16}, which is a strengthening of Theorem \ref{incidence}.

\begin{theorem}[Incidences between points and algebraic curves, \cite{SZ16}]\label{SharirZahl} 

Let $P$ be a set of $m$ points in the plane. Let $\mathcal{C}$ be a set of $n$ algebraic plane curves of degree at most $D$,  no two of which share a common irreducible component.  Assume that we can parameterize these
curves using $s$ parameters. Then for any $\varepsilon > 0$,
the number $I(P, \mathcal{C})$ of incidences between the points of $P$ and the curves of $\mathcal{C}$ satisfies
$$I(P, \mathcal{C}) = O\left(m^{\frac{2s}{5s-4}}n^{\frac{5s-6}{5s-4}+\varepsilon}+m^{2/3}n^{2/3} + m + n\right).$$
\end{theorem}

Now we are ready to prove our main result.

\begin{theorem}\label{main}
For every fixed $\varepsilon >0$, the maximum number of triangles of unit area satisfies  $$f(n)= O(n^{\frac{9}{4}+\varepsilon}).$$
\end{theorem}

\begin{proof}
Consider the additive combinatorial equivalent form of the problem in Equation \ref{xytriple}, take the solution set with maximum number of solutions and denote it by $H$. For every ordered pair $(x_i, y_i), (x_j, y_j)$ where $x_i<x_j$, the solutions of Equation (\ref{xytriple}) are points $(x_k, y_k)$ of a bounded degree rational curve defined by (\ref{xytriple}), with the condition that $x_j<x_k$ must hold. Hence we obtain at most $\binom{n}{2}$ plane curves belonging to an $s=4$-dimensional family, as the family depends on the real values $\{x_i, y_i, x_j, y_j\}$. One can verify also rather easily  that no two of these curves share a common irreducible component. This can be done either directly, by deducing  from Equation  (\ref{xytriple}) an equivalent reformulation, a degree $2$ polynomial   in variables $x_k, y_k$ where the coefficients are (polynomial) functions of   $x_i, y_i, x_j, y_j$, or by referring to  Lemma \ref{hyperbolas} which provides a description of the locus of the points $(x_k, y_k)$  satisfying (\ref{xytriple})  in terms of $\{x_i, y_i, x_j, y_j\}$. Hence, applying the result of Sharir and Zahl (Theorem \ref{SharirZahl}) we get the desired bound.
\end{proof}


The lower bound for $f(n)$ follows from the results in the next section, by scaling the triangles of minimum area to have area $1$.

\section{Number of maximum and minimum area triangles, bounds on $m(n)$ and $M(n)$}

\subsection{Minimum area triangles}

In this subsection we prove Theorem \ref{minim} by determining the maximal possible number of triangles of minimal area constituted by $n$ lines, up to a factor of $2$. This will follow from the results on the lower and upper bound below.

\begin{prop}\label{minupper}  $m(n)\leq \lfloor n(n-2)/3\rfloor$ for every $n$ and $m(n)\le \lfloor n(n-2)/3 \rfloor-1$ if $n\equiv 0,2 \pmod 6$.
\end{prop}	

\begin{proof}
	Observe that if a triangle is of minimal area, then none of the lines can intersect its sides. Hence the maximal number of triangles of minimal area is at most the number of triangular faces $K(n)$ in an arrangement of n lines. The latter problem became famous as the so-called Tokyo puzzle or the problem of Kobon triangles. The best bound is by Bader and Clément \cite{BC07}, who showed that $K(n)\le \lfloor n(n-2)/3 \rfloor$ for every $n$ and $K(n)\le \lfloor n(n-2)/3 \rfloor-1$ if $n\equiv 0,2 \pmod 6$.
\end{proof}

The bound on $K(n)$ is almost sharp since Füredi and Palásti constructed a general arrangement to prove  $K(n)\geq \lfloor n(n-3)/3\rfloor$ \cite{FP84}. See also the construction of Forge and Ramírez-Alfonsín  \cite{FRA98}.

\begin{prop}\label{minlowerhex}  Suppose that $n\ge 3$.   \begin{equation*}
m(n)\geq  \begin{cases}   {6{l}^2 } &\text{ if $n=6l$},\\  {6{l}^2 + 2jl+j-2} &\text{ if $n=6l+j$, $1\le j\le 5$.}
\end{cases}
\end{equation*}
\end{prop}

\begin{proof} Take the grid depicted in Figure \ref{hexgrid}. Choose $n$ lines such that they are as close to the center of a hexagonal face as possible. If there are 2, 3 or 4 lines in the outermost layer, pick these to be in consecutive clockwise position. 

Assume that n=6l. Add the lines to the diagram layer by layer, starting from the center. Adding the six lines of the $i$-th layer will create $6(2i-1)$ new triangular faces ($6i$ outside the hexagon formed by those lines and $6(i-1)$ inside it). Therefore the total number of triangular faces is $\sum_{i=1}^l 6(2i-1)=6l^2$. If $n=6l+j$ the result follows similarly with elementary counting on the outermost layer.
\end{proof}

	\begin{figure}[!htp]
		\centering
		\includegraphics[width=0.4\linewidth]{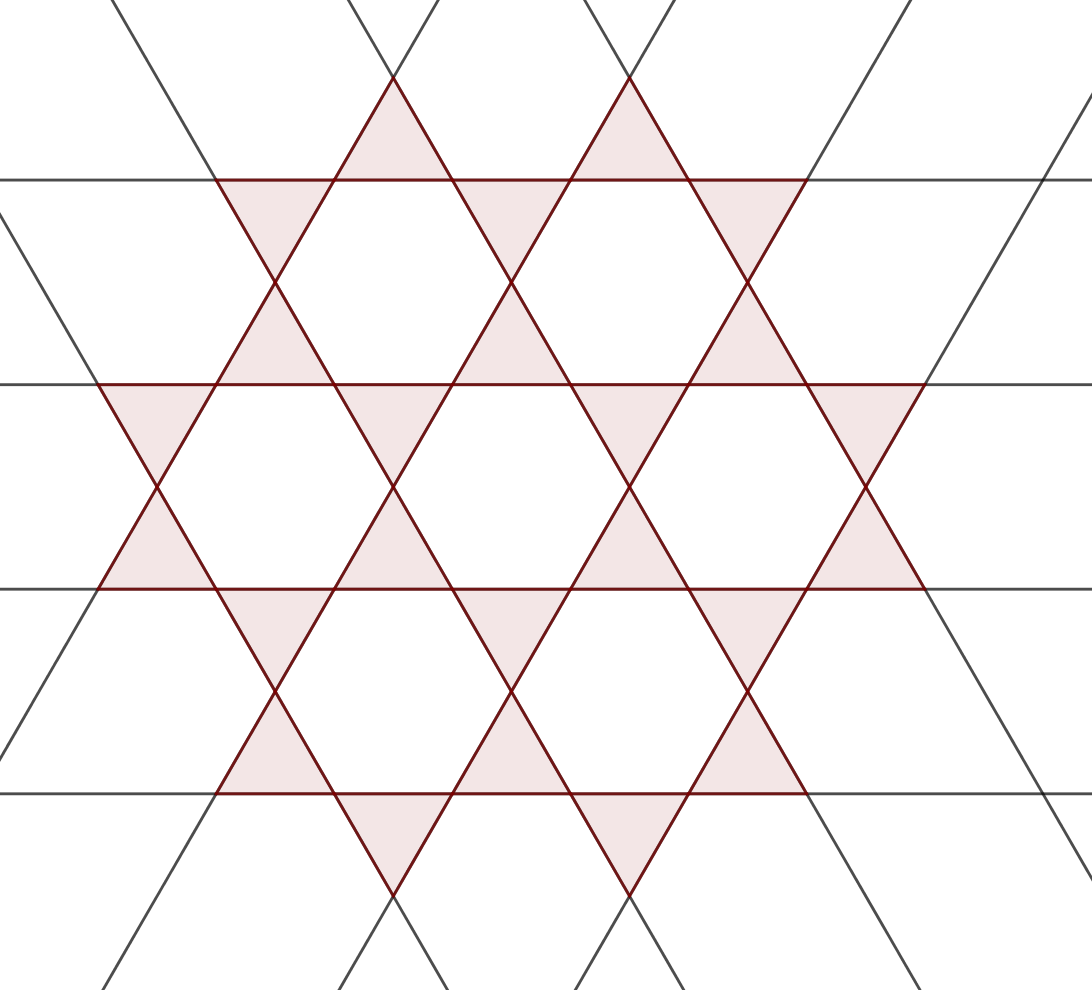}
		\caption{A hexagonal grid formed by 12 lines}
		\label{hexgrid}
	\end{figure}

\begin{conj}\label{min_sejt} The lower bound of Proposition \ref{minlowerhex} is sharp if $n$ is large enough.
\end{conj}

Note that these lower bounds are not met if $n$ is small. C. T. Zamfirescu \cite{Zam17} recently proved that even the number of facial \emph{congruent} triangles exceeds this bound if $n\leq 12$, see Table \ref{tab:Zamf}. On the other hand, the construction described in Proposition \ref{minlowerhex} provides a general lower bound as well for the number of facial \emph{congruent} triangles in terms of the number of lines, which exceeds the bound of C.T. Zamfirescu if $n$ is large.

\begin{table}[h!]
\centering
\begin{tabular}{l|r|r| r|r|r|r|r|r|r|r}
\quad \quad \# of lines, $n$ & $3$ & $4$ & $5$ & $6$ & $7$ & $8$ & $9$ & $10$ & $11$ & $12$\\\hline
\shortstack{\# of congruent facial \\  triangles, lower bound} & $1$ & $2$ & $5$ & $6$ & $\geq9$ & $\geq 12$ & $\geq 15$ & $\geq 20$ & $\geq 23$ & $\geq 26$ \\ \hline
\shortstack{\# of congruent triangles,\\ lower bound via Prop. \ref{minlowerhex}}  &  $1$ & $2$ & $3$ & $6$ & $7$ & $10$ & $13$ & $16$ & $19$ & $24$ \\ \hline
\shortstack{\# of congruent triangles,\\ lower bound via Prop. \ref{minlower}}  &  $0$ & $1$ & $2$ & $4$ & $6$ & $8$ & $12$ & $14$ & $18$ & $22$ \\ \hline

\end{tabular}
\caption{ Comparison of the constructions for the number of congruent or minimal area triangles in small cases}\label{tab:Zamf}
\end{table}

We can obtain the same order of magnitude in an essentially different way as well. 

\begin{prop}\label{minlower}  Suppose that $n\ge 3$. Then \begin{equation*}
m(n)\geq  \begin{cases}   {6{l}^2 + 2jl}-2 &\text{ if $n=6l+j$,  $j\in \{0, \pm1, \pm2\}$},\\  {6{l}^2 + 6l} &\text{ if $n=6l+3$.}
\end{cases}
\end{equation*}
\end{prop}	

\begin{proof} Take a triangular grid. If $n\equiv 3 \pmod 6$, choose those $n$ lines of the grid which are the closest to a fixed point on the grid.  If $n\not\equiv 3 \pmod 6$, choose those $n$ lines of the grid which are the closest to a fixed point which is the center of a triangle in the grid. A simple inductive argument similar to the one in Proposition \ref{minlowerhex} shows that the number of constructed facial triangles equals the desired quantity, see Figure \ref{fig:grid}.
\end{proof}

	\begin{figure}[!htp]
		\centering
		\includegraphics[width=0.4\linewidth]{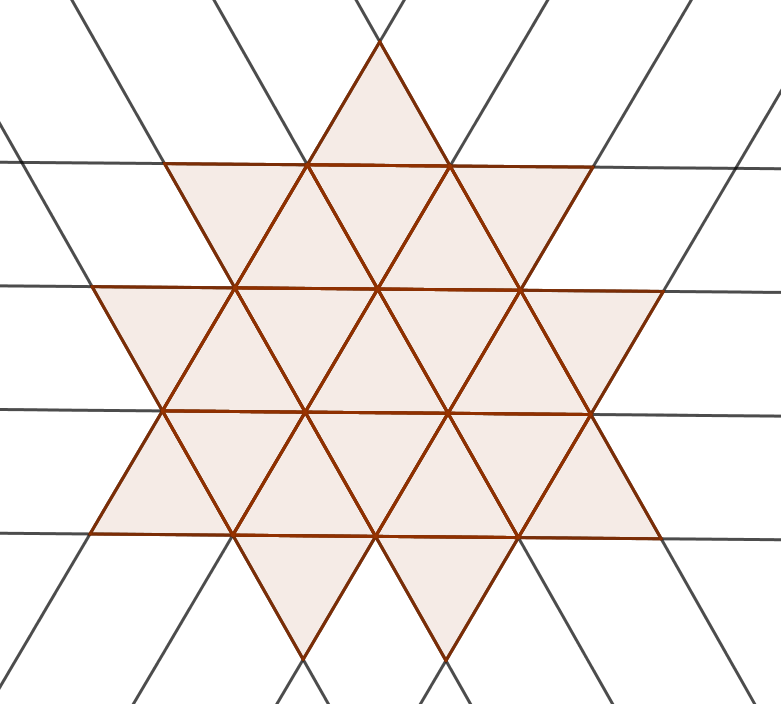}
		\caption{A triangular grid formed by 12 lines}
		\label{fig:grid}
	\end{figure}

 The above constructions differ in several aspects. Firstly, the former one does not contain concurrent triples of lines. Secondly, note that in the  upper bound on $K(n)$ of  Bader and Clément \cite{BC07} a key observation was that every line segment between consecutive intersections on a line belongs to at most one triangular region. This property appears only in the former construction.
 Thus if Conjecture \ref{min_sejt} holds, it would imply that arrangements which almost attain the extremum may have significantly different structure. 

\subsection{Maximum area triangles}

We start with a construction  to prove the lower bound of Theorem \ref{maxi} on the number of maximum area triangles. The main idea is the following. Suppose you have a construction with some number of maximal area triangles. Then we can add a new line that doesn't create large triangles, i.e. the maximal area doesn't increase. We can slide this line until it creates an extra triangle of maximal area. This way we can create a new maximal area triangle per line. To improve this we will show that we can add five lines together to get seven new maximal area triangles. Five of the new maximal triangles will appear between these five new lines and then by sliding the five lines together we will get  two extra ones.  

The precise construction requires a couple of lemmas first. 

\begin{prop}\label{intersect}
Let $ABC$ be one of the maximal area triangles in the arrangement and let $\ell$ be one of the lines of the arrangement. Then either $\ell$ intersects the interior of $ABC$ or it is parallel to one of the sides of $ABC$. 
\end{prop}

\begin{proof}
If $\ell$ is not parallel to one of the three sides then it intersects each of the three lines. Suppose $\ell$  avoids the interior of the triangle. By symmetry we can assume that $\ell$ runs as in Figure \ref{fig:avoid}. Then $A'BC'$ is a triangle of larger area which contradicts the maximality of $ABC$.
\end{proof}
\begin{figure}[!htp]
    \centering
    \begin{subfigure}[b]{0.45\textwidth}
        \includegraphics[width=\textwidth]{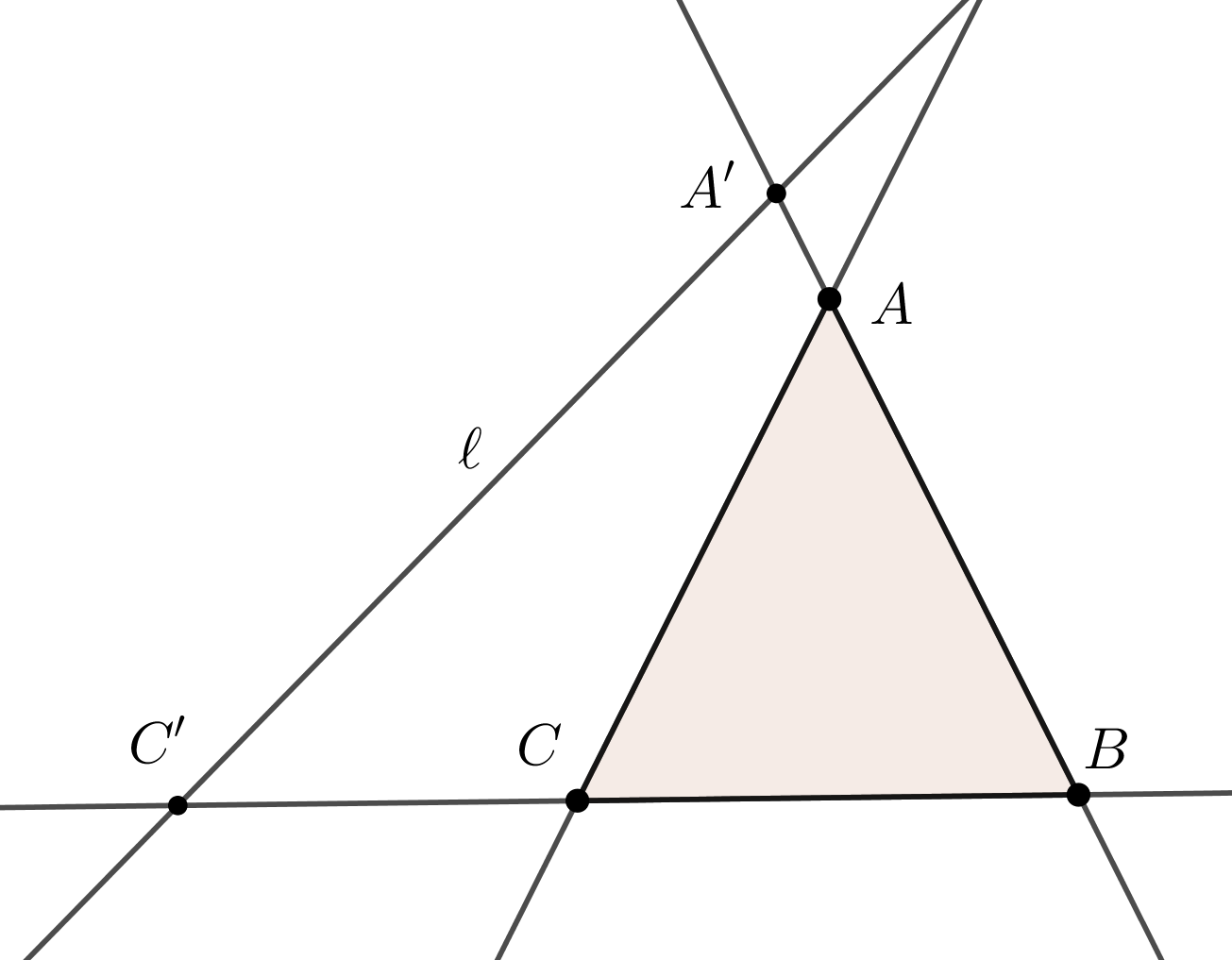}
        \caption{A line avoiding a maximal triangle}
        \label{fig:avoid}
    \end{subfigure}
    \begin{subfigure}[b]{0.45\textwidth}
        \includegraphics[width=\textwidth]{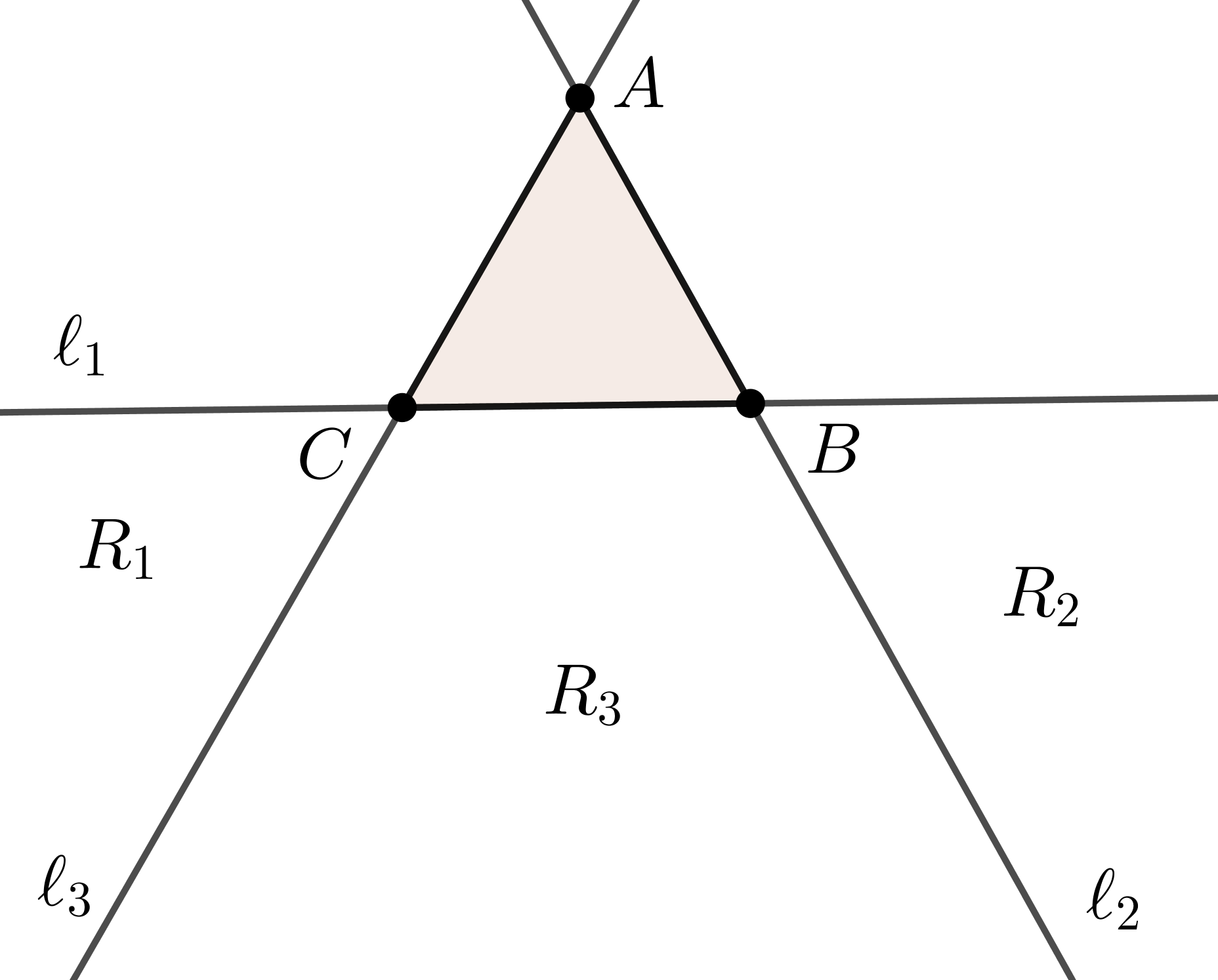}
        \caption{Regions around a triangle}
        \label{fig:regions}
    \end{subfigure}
    \caption{Line positions with respect to a maximum area triangle}
\end{figure}

\begin{prop}\label{sameside}
Suppose that there are no parallel lines in the arrangement and that the triangle $\Delta$, formed by lines $(\ell_1,\ell_2,\ell_3)$, is a maximal area triangle. Then all the maximal area triangles that are supported by $\ell_1$ lie on the same side of $\ell_1$. 

\end{prop}

\begin{proof}
Suppose that a triangle $\Delta'$, formed by the lines  $(\ell_1,\ell_4,\ell_5)$, is also a maximal area triangle and it lies on the opposite side of $\ell_1$. Let $P=\ell_4\cap \ell_5$ and consider the possible positions of $P$. We will denote the three regions by $R_1,R_2$ and $R_3$ as seen in Figure \ref{fig:regions}. 

Assume that $P$ lies in the interior of $R_1\cup R_3$. By Proposition \ref{intersect} we know that $\ell_4$ and $\ell_5$ must intersect the interior of the triangle $ABC$, therefore they intersect $\ell_1$ on the interior of the $\overrightarrow{BC}$ ray. But then the line $\ell_2$ avoids the maximal triangle $\Delta'$, contradicting Proposition \ref{intersect}. Similarly $P$ cannot lie in the interior of $R_2\cup R_3$.\end{proof}

\begin{prop}
If there are no parallel lines in an arrangement then we can add a new line $\ell$ to the arrangement such that it supports no maximal area triangle in the new arrangement.  
\end{prop}

\begin{proof}\label{nonewmax}
Pick an arbitrary direction that is not parallel to any of the lines of the arrangement. Choose $\ell$ to be the line that has the chosen direction and for which the largest new triangle area created is the smallest possible. Let's say this area is $q$. Then $\ell$ must support two triangles on opposite sides that have area $q$. Otherwise we could translate $\ell$ slightly in one direction to decrease all the new areas below $q$. By Proposition \ref{sameside} this implies that the $q$ cannot be the maximal area in the whole arrangement.
\end{proof}

\begin{prop}\label{rectangle}
If there are no parallel lines in an arrangement then we can find a rectangle $ABCD$ such that if we add any line to the arrangement that intersects both $AB$ and $CD$ we create no new maximal area triangles.
\end{prop}

\begin{proof}
By Proposition \ref{nonewmax} we can find a line $\ell$ that creates no new maximal area triangles. Let $\ell'$ be a line parallel to $\ell$ which also doesn't create a new maximal triangle and lies so close to $\ell$ that no two line of the arrangement intersects each other between $\ell$ and $\ell'$.
Then any line $f$ that intersects all lines of the arrangement between $\ell$ and $\ell'$ doesn't create a new maximal area triangle. This follows from the fact that if $f$ supports a triangle then either $\ell$ or $\ell'$ avoids that triangle, so by Proposition \ref{intersect} the triangle cannot be maximal. Then we can choose points $A,D$ on $\ell$ and $B,C$ on $\ell'$ appropriately, see Figure \ref{fig:abcd}.  
\end{proof}

\begin{figure}[!htp]
    \centering
    \begin{subfigure}[b]{0.36\textwidth}
        \includegraphics[width=\textwidth]{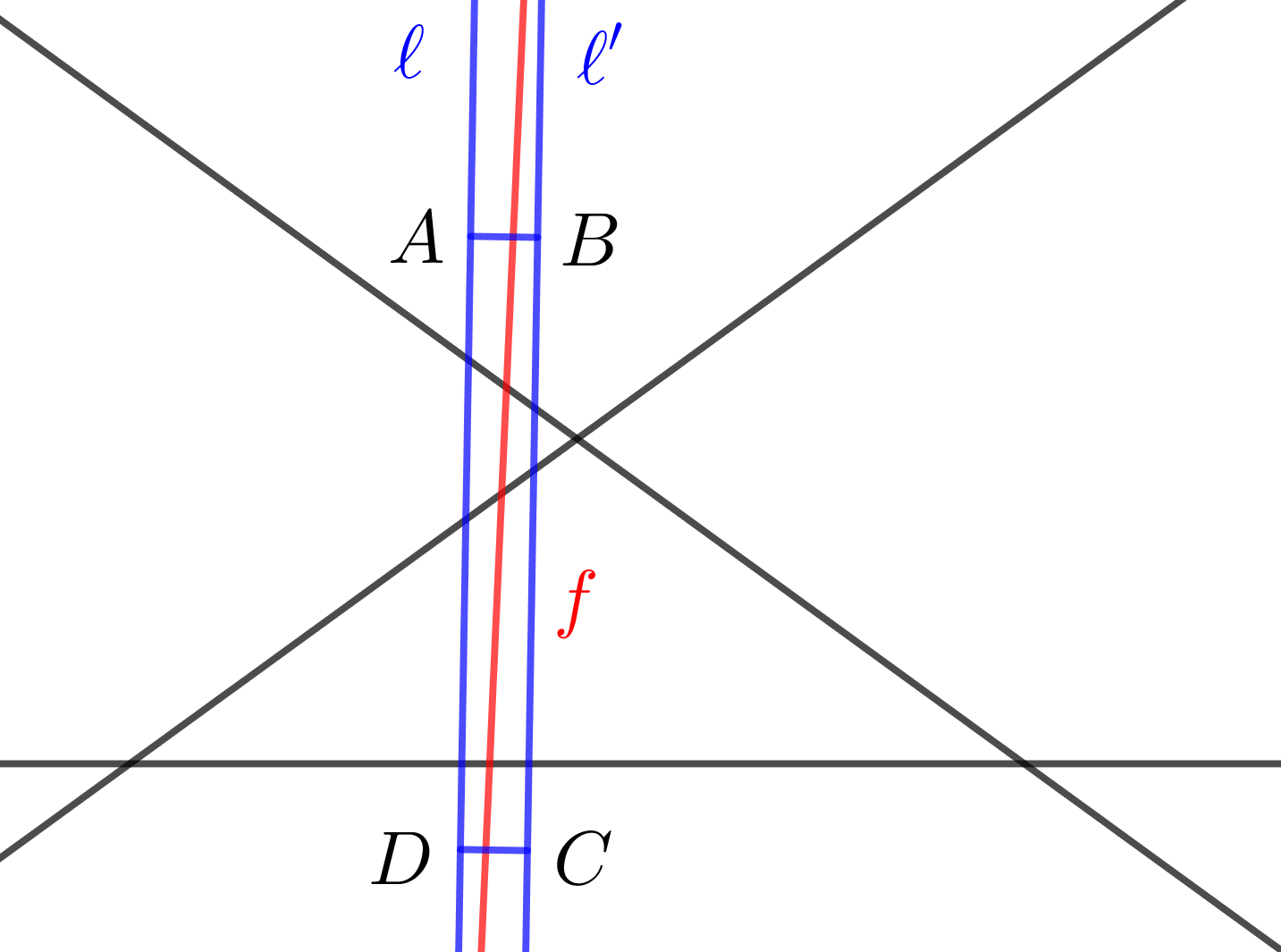}
        \caption{No new maximal triangles.}
        \label{fig:abcd}
    \end{subfigure}
    \begin{subfigure}[b]{0.38\textwidth}
        \includegraphics[width=\textwidth]{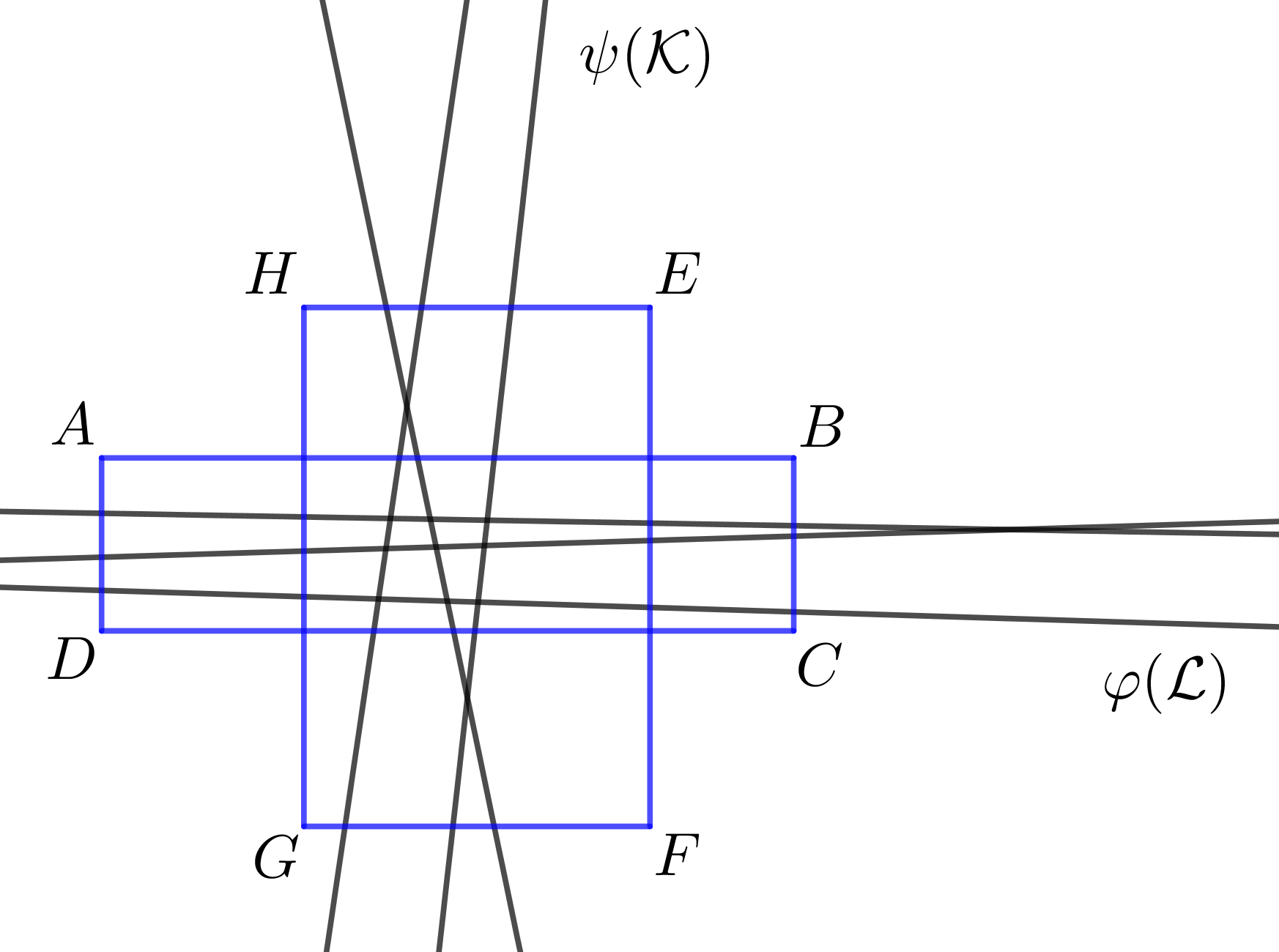}
        \caption{Combining two constructions}
        \label{fig:konstr}
    \end{subfigure}
    \begin{subfigure}[b]{0.20\textwidth}
        \includegraphics[width=\textwidth]{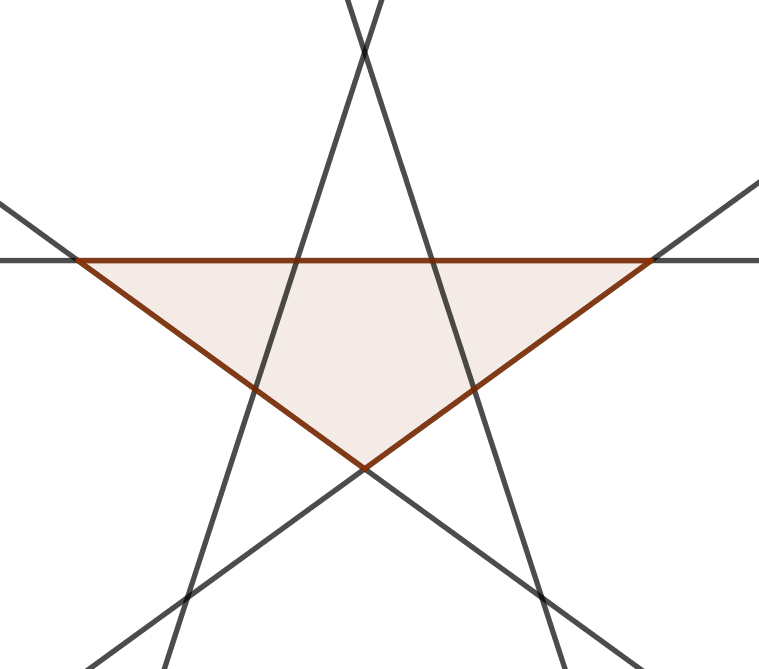}
        \caption{Pentagon}
        \label{fig:penta}
    \end{subfigure}
    \caption{Ingredients for the recursive construction}
\end{figure}

For an arrangement $\mathcal{L}$ let $T(\mathcal{L})$ denote the number of maximal area triangles. For example $T(\mathcal{L})=5$ if $\mathcal{L}$ consists of five lines forming a regular pentagon. For an affine transformation $\varphi$ let $\varphi(\mathcal{L})$ denote the image of $\mathcal{L}$.  

\begin{prop}\label{combine}
If $\mathcal{L}$ and $\mathcal{K}$ are arrangements of lines that contain no parallel lines then there exist affine transformations $\varphi$ and $\psi$ such that $T(\varphi(\mathcal{L})\cup \psi(\mathcal{K}))\ge T(\mathcal{L})+T(\mathcal{K})+2$.
\end{prop}

\begin{proof}
We can assume that the maximal area triangles have the same area in  $\mathcal{L}$ and $\mathcal{K}$. Using Proposition \ref{rectangle} we can define rectangle $ABCD$ for $\mathcal{L}$ and rectangle $EFGH$ for $\mathcal{K}$. Then applying an area preserving affine transformation we can place the two construction such that the two rectangles cross each other (see Figure \ref{fig:konstr}). Now every line of $\varphi(\mathcal{L})$ crosses $EF$ and $GH$ and every line of $\psi(\mathcal{K})$ crosses $AB$ and $CD$. By Proposition \ref{rectangle} this means that in the new construction the maximal area triangles are the same as they are in $\mathcal{K}$ and $\mathcal{L}$. So we have exactly $T(\mathcal{L})+T(\mathcal{K})$ maximal triangles.\\ Finally we increase this number by two in two steps.  Translate first the lines of $\varphi(\mathcal{L})$ together in an arbitrary direction until a new maximal area triangle appears, formed by lines both from the translates of $\varphi(\mathcal{L})$ and  $\psi(\mathcal{K})$. We may assume that only one such triangle $\Delta^*$ is formed, and it has exactly one supporting line $\ell^*$ in   $\psi(\mathcal{K})$. Now if we translate again  the lines of $\varphi(\mathcal{L})$, this time along the line $\ell^*$, then obviously neither the area of triangles formed by the lines from $\varphi(\mathcal{L})$ or $\psi(\mathcal{K})$, nor the area of $\Delta^*$ will change. However, some translated lines of   $\varphi(\mathcal{L})$ will eventually form yet another triangle of maximal area together with some lines from $\psi(\mathcal{K})$.
\end{proof}

It is easy to see that the lower bound of Theorem \ref{maxi} follows. We start with five lines forming a regular star pentagon (see Figure \ref{fig:penta}). Then we use Proposition \ref{combine} repeatedly, always using the previous construction as $\mathcal{L}$ and  five lines forming a regular pentagon as $\mathcal{K}$. 

\begin{theorem}\label{2/3}
$M(n)\le \frac{2}{3}n(n-2)$.
\end{theorem}

\begin{proof}
We will show that in an arrangement of $n$ lines, any fixed line $\ell$ supports at most $2(n-2)$ triangles of maximal area. This immediately implies the statement of the theorem.

Let $\ell$ be a fixed line in the arrangement. We may assume that all other lines intersect it as otherwise they would not form any triangle together. Consider $\ell$ as the $x$ axis of a coordinate system, and let $x_i$ denote the $x$ coordinate of the intersection of $\ell$ and $\ell_i$ for all $i=1,2,\dots, n-1$. We also use the notation $y_i$ for the cotangent of the (directed) angle determined by $\ell$ and $\ell_i$. By Lemma \ref{areaoftri}, the area of the triangle $T_{ij}$ determined by the lines $\ell$, $\ell_i$ and $\ell_j$ is $\text{Area}(T_{i,j})=\frac{(x_i-x_j)^2}{2|y_i-y_j|}$. (If $x_i=x_j$ or $y_i=y_j$ then there is no triangle to speak of.) If the sign of $x_i-x_j$ and $y_i-y_j$ is the same, then the triangle is located under $\ell$, otherwise it is located over it.

Without loss of generality, we may assume  that the maximal triangle area is $1/2$. Then $(x_i-x_j)^2\le |y_i-y_j|$ applies to all pairs $(i,j)$, with equality if and only if $\text{Area}(T_{i,j})$ is maximal.

Let us define the graph $G_\ell^+$, and resp. $G_\ell^-$ on the vertex set $\{v_1, v_2,\dots, v_{n-1}\}$ and connect $v_i$ to $v_j$ if $(x_i-x_j)^2=|y_i-y_j|$ and the sign of $x_i-x_j$ and $y_i-y_j$ is the same or respectively, the opposite. We will show that there is no cycle in $G_\ell^+$, therefore $|E(G_\ell^+)|\le n-2$ holds for the cardinality of the edge set. The same argument applies to $G_\ell^-$ as well, yielding $|E(G_\ell^-)|\le n-2$. Therefore the total number of edges, which is equal to the number of triangles of maximal area supported by $\ell$, is at most $2(n-2)$.

Assume that there is a cycle $v_{i_1}v_{i_2}\dots v_{i_k}$ in $G_\ell^+$. We will get a contradiction using two simple propositions, where we consider the indexing of the vertices modulo $k$.

\begin{proposition}\label{obs1}
The signs of $x_{i_t}-x_{i_{t+1}}$ and $x_{i_{t+1}}-x_{i_{t+2}}$ are the opposite.
\end{proposition}

\begin{proof}
Assume that the signs are the same. Then $$|y_{i_{t+2}}-y_{i_{t}}|=|y_{i_{t+2}}-y_{i_{t+1}}|+|y_{i_{t+1}}-y_{i_{t}}|=
(x_{i_{t+1}}-x_{i_{t+2}})^2+(x_{i_t}-x_{i_{t+1}})^2<(x_{i_{t}}-x_{i_{t+2}})^2$$ 
would hold, a contradiction.
\end{proof}

\begin{proposition}\label{obs2}
There are no four vertices $v_a, v_b, v_c$ and $v_d$ in $G_\ell^+$ such that $x_a<x_b<x_c<x_d$ and $v_av_c, v_bv_c, v_bv_d\in E(G_\ell^+)$.
\end{proposition}

\begin{proof}
Assume that there are four such vertices. Then
$$(x_d-x_a)^2\le y_d-y_a=(y_d-y_b)+(y_c-y_a)-(y_c-y_b)=(x_d-x_b)^2+(x_c-x_a)^2-(x_c-x_b)^2$$
After rearranging, we get
$$x_ax_c+x_bx_d \le x_ax_d+x_bx_c,$$
which can be written as $(x_a-x_b)(x_c-x_d)\le 0$, a contradiction.
\end{proof}

Returning to the cycle $v_{i_1}v_{i_2}\dots v_{i_k}$, Proposition \ref{obs1} implies that $k$ is even. We may assume that $|x_{i_1}-x_{i_2}|>|x_{i_2}-x_{i_3}|$ after shifting the indexing of the vertices if necessary. This means that $x_{i_3}$ is between $x_{i_1}$ and $x_{i_2}$.

Proposition \ref{obs1} tells us that $x_{i_4}$ must be in the same direction from $x_{i_3}$ as $x_{i_2}$. However Proposition \ref{obs2} implies that it can't be past $x_{i_2}$. Note that $x_{i_2}=x_{i_4}$ is also impossible since this would imply $y_{i_2}=y_{i_4}$ and $\ell_{i_2}=\ell_{i_4}$. Therefore $x_{i_4}$ must be between $x_{i_2}$ and $x_{i_3}$.

Following this argument, we find that $x_{i_{t+2}}$ must be between $x_{i_t}$ and $x_{i_{t+1}}$ for all $t=1, 2,\dots, k-2$. Then the vertices $v_1, v_{k-1}, v_k, v_{k-2}$ violate Proposition \ref{obs2}, a contradiction.
\end{proof}

\begin{remark} Theorem \ref{2/3} can be even strengthened, as   $M(n)\leq \frac{1}{3}n(n-1)$ also holds. Indeed, one can verify that Proposition \ref{sameside} is true in a more general form, namely if there are parallel lines in the line arrangement, then there may exist maximal area triangles on both sides of a fixed line $\ell$, but on one of the sides there is no more than one maximal area triangle. This result yields $|E(G_\ell^-)|+ |E(G_\ell^+)|\le n-1$ in the proof above, implying our stated improvement. The details are left to the interested reader.
\end{remark}

\section{Lines defining distinct area triangles}

In this section we assume that the lines in the original arrangement are in general position. More specifically, we will require that no six of them are tangent to a common quadratic curve on the plane.

To prove Theorem \ref{disti}, we begin with the following result. 

\begin{lemma}
\label{hyperbolas}
Let $r_1$ and $r_2$ be two rays from a  point $O$ and $\lambda \in \mathbb{R}^+$  fixed. Then those lines that form a triangle with $r_1$ and $r_2$ of area $\lambda$ are all tangent to a fixed hyperbola. The two rays belong to the asymptotes of this hyperbola.
\end{lemma}

\begin{proof}
Affine transformations preserve lines, conics and ratios of areas. Therefore, we may assume that $r_1$ and $r_2$ are perpendicular and correspond to the positive parts of the $x$ and $y$ axis, respectively.

Now, for a positive real number $c$ consider the hyperbola $xy=c$ and $(x_1,y_1)$ any point on it. The tangent $t$ at $(x_1,y_1)$ is given by the equation $xy_1+yx_1=2c$. Let $P_1=(2c/y_1,0)$ and $P_2=(0,2c/y_2)$ be the intersections of $t$ and the $x$ and $y$ axis, respectively. Then the area of the triangle $OP_1P_2$ is $\frac{4c^2}{2x_1y_1}=2c$.

Any line that intersects the positive parts of the $x$ and $y$ axis must be tangent to exactly one of these hyperbolas, and as seen above the area of the triangle it defines depends completely and injectively on $c$. Therefore, triangles with the same area must all be tangent to a fixed one of these hyperbolas.
\end{proof}

From here, we deduce the following result.

\begin{cor}
\label{twenty}
Let $\ell_1$ and $\ell_2$ be two intersecting lines. Then for any fixed value $\lambda \geq 0$, there can be at most $20$ lines in general position such that each of them forms a triangle with $\ell_1$ and $\ell_2$ of area $\lambda$.
\end{cor}

\begin{proof}
    Note that $\ell_1$ and $\ell_2$ define four quadrants. If $\lambda>0$, each line defines a triangle of positive area with $\ell_1$ and $\ell_2$, so it intersects both rays of one of the quadrants. By Lemma \ref{hyperbolas}, we can have at most $5$ lines per quadrant, so we obtain at most $20$ lines.
    
    If $\lambda=0$, then each line has to go through the point of intersection of $\ell_1$ and $\ell_2$. There can be at most $5$ of these lines, as otherwise the intersection point would be a common degenerate conic tangent to $6$ lines.
\end{proof}

The second ingredient that we use is a rainbow Ramsey result. We apply the following particular version of a result proven by Conlon et al. \cite{Con} and independently by Martínez-Sandoval, Raggi and Roldán-Pensado \cite{MRR}. It has been used before to obtain similar results in combinatorial geometry in which a large structure with distinct substructures is desired.

\begin{theorem}
\label{rainbow}
Let $H$ be an $m$-uniform hypergraph on the vertex set $V$ and $k$ a positive integer. Suppose that the hyperedges of $H$ are coloured in such a way that no $2$ vertices lie in $k$ edges of the same color.

Then there exists a set of $$\Omega_k(n^{1/(2m-1)})$$ vertices for which all the hyperedges receive distinct colors. 

\end{theorem}

We are ready to prove the main result of this section.

\begin{proof}[Proof of Theorem \ref{disti}]
Note that there cannot be six or more lines with the same slope, as otherwise a seventh line would be a degenerate conic tangent to all of them at infinity. Therefore, the lines define at least $\frac{n}{5}$ distinct slopes, and taking at most one for each slope  we may extract a subset $L'$ of size at least $\frac{n}{5}$ so that no two lines of $L'$ are parallel.

Consider the complete $3$-uniform hypergraph $H$ whose vertex set is $L'$. Since no two lines of $L'$ are parallel, we may provide a colouring of the $3$-edges of $H$ by assigning to each triple the area of the triangle it defines. By Corollary \ref{twenty}, no pair of vertices belongs to $21$ or more triples of the same colour. Therefore, by Theorem \ref{rainbow} we obtain a set of $\Omega((n/5)^{1/5})=\Omega(n^{1/5})$ lines such that the triangles that they define have all distinct areas.
\end{proof}

\section{Discussion and open problems}

One could also raise here an analogue question to the well known problem due to Erdős, Purdy and Strauss, which is formulated as 

\begin{problem}[Erdős, Purdy, Straus, \cite{EPS}]
Let $S$ be a set of $n$ points in $\mathbb{R}^d$ not all in one hyperplane. What is
the minimal number  of distinct volumes of non-degenerate simplices with vertices in $S$?
\end{problem}

Concerning the case $d=2$, we refer to e.g. \cite{dumtoth07} and its reference list.
Note that to obtain reasonable results on the cardinality of \emph{distinct areas}, one has to prescribe certain restrictions to avoid huge classes of parallel lines hence obtaining only few triangles. However, having assumed e.g. that no pair of parallel lines appear, the distribution of the areas can change significantly. We conjecture that  the  number of the unit area triangles drops to $O(n^2)$ in that case, and in fact we could not even  find evidence of that the order of magnitude is  $\Omega(n^2)$. 

The proof of the upper bound on the number of maximum area triangles was relying on an argument  about maximum  area triangles sharing a common line $\ell$ that provides a linear upper bound. Although it is easy to see that a linear lower bound is realisable by a set of $n-2$ tangent and two asymptotes of a hyperbole branch (see Figure \ref{figure4}), we conjecture that this won't provide the right (quadratic) order of magnitude for $M(n)$. Note that this phenomenon appeared concerning the unit area triangles as well when we compared $g(n)$ and $f(n)$, see Section 2.\\

\begin{figure}[h]
\begin{center}
\includegraphics[scale = 0.7] {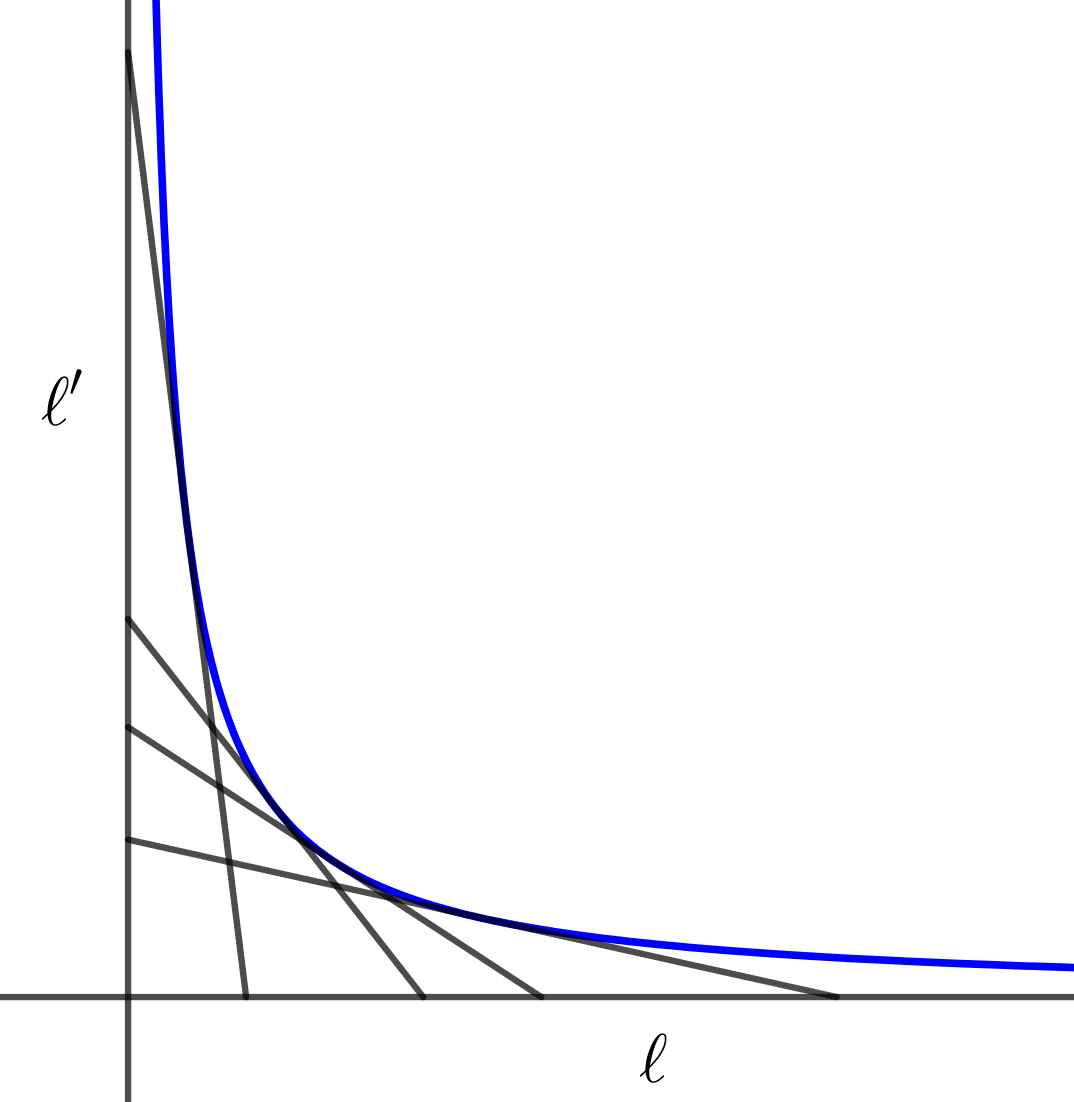}
\end{center}
\caption{Maximum area triangles lying on $\ell$, formed by $\ell, \ell'$, and a third line from the tangent line set.}
\label{figure4}\end{figure}

In fact, we believe that the following holds.

\begin{conj} The order of magnitude of  $M(n)$, largest possible number of triangles with maximum area in arrangements of $n$ planar lines is $O(n^{1+\varepsilon})$ for every $\varepsilon>0$.
\end{conj}

In general we have seen that in these types of combinatorial geometry problems,  small (or minimum) distances (or areas) may occur much more frequently than large (or maximum) distances (areas).

Supposing that this assertion holds, it raises yet another interesting inverse research problem from a statistical point of view.

\begin{problem}\label{distrib}
Suppose that fr each $n$, a set of  $n$ lines are given in a Euclidean plane.  Assume  that the number of triangles determined by a triple of lines which have unit area is $\phi(n)$, where $\phi(n)/ n\rightarrow \infty$. Prove a lower bound (in terms of $\phi(n)$) on the number of triangles having area greater than $1$.
\end{problem}

The analogue of Problem \ref{distrib} for the original Erdős-Purdy problem on  distances in a planar point set seems also widely open. Some results were obtained by Erdős, Lovász and Vesztergombi
 \cite{Veszterg}.

We also note that the problem may be  investigated in  a finite field setting as well, similarly to  \cite{finite}.

In Theorem \ref{disti}, we assume that no six lines are tangent to a common conic. This implied, in particular, that no six lines are pairwise parallel. A more natural condition would be to simply require that no two lines are parallel. We were not able to obtain any non-trivial bounds under this hypothesis.

\begin{problem}\label{distrib2}
What is the maximum number $D'(n)$ such that in any arrangement of $n$ lines on the plane, no two of them parallel or three through a common point, there are $D'(n)$ lines that form triangles of distinct areas?
\end{problem}

The bound in Theorem \ref{disti} can be improved by a logarithmic factor, as mentioned in \cite{Con}. The problem could also be generalized to higher dimensions as follows.

\begin{problem}\label{distrib2}
A set of $n$ hyperplanes in general position are given in $\mathbb{R}^d$. What is the maximum number $D_d(n)$ such that we can always find a subset of these hyperplanes of this size for which all the simplices that they define have distinct $d$-dimensional volume?
\end{problem}

We finish by mentioning that there are very few geometric problems with this combinatorial flavour in which the bounds are asymptotically tight. A related question concerning circumradii is discussed in \cite{MR}.

\textbf{Acknowledgement} We are grateful to  the anonymous referee whose remarks helped us to improve the presentation of the paper and Nóra Frankl for the fruitful discussion on the topic.

\end{document}